\documentclass[a4paper,12pt]{article}

\usepackage{amsmath,amssymb,amsthm,here,bm,framed,arydshln,fancyhdr}
\usepackage{comment}
\usepackage[all]{xy}
\usepackage[right=2.5cm,left=2.5cm,top=3cm,bottom=3cm]{geometry}
\usepackage{tikz}
\usetikzlibrary{patterns}
\usepackage{breqn}

\usepackage{docmute}

\newtheorem{theorem}{Theorem}[section]
\newtheorem{lemma}{Lemma}[section]
\newtheorem{proposition}{Proposition}[section]

\newtheorem*{problem*}{Problem}

\theoremstyle{definition}
\newtheorem{definition}{Definition}[section]
\newtheorem{remark}{Remark}[section]

\DeclareMathOperator{\tr}{tr}

\numberwithin{equation}{section}

\makeatletter
  	
	\@addtoreset{figure}{section}
\title{\bf Unitary monodromies of rank two Fuchsian systems with $(n+1)$ singularities}   
\author{Shunya Adachi} 
\date{}
\begin{document}
%%%%%%%%%%%%%%%%
\maketitle
%
%%%%%%%%%%%%%%
%
\begin{abstract}
We study the unitarity of monodromies of rank two Fuchsian systems of SL type with $(n+1)$ regular singularities on the Riemann sphere,
namely, we give a sufficient and necessary condition for the monodromy group to be conjugate to a subgroup of a special unitary group $\mathrm{SU}(p,q)$. 
When $n\ge 3$, the moduli space of irreducible monodromies can be realized as an affine algebraic set in $\mathbb{C}^m$ for some $m \in \mathbb{N}$.
In this paper, we give a characterization and construction of unitary monodromies in terms of this affine algebraic set.
The signatures of unitary monodromies are also classified.



\end{abstract}

%%%%%%%%%%%%%%%%
%%%%%%%%%%%%%%%%
\section{Introduction}
\label{Section_introduction}
%%%%%%%%%%%%%%%%
%%%%%%%%%%%%%%%%
Consider a rank two Fuchsian system of SL type  
\begin{equation}\label{Eq}
\frac{dU}{dx}=\left(\sum_{j=1}^n \frac{A_j}{x-t_j}\right)U \qquad 
(A_j \in \mathfrak{sl}(2,\mathbb{C})).
\end{equation}
We assume
$
 A_{n+1}:=-\sum_{j=1}^n A_j \neq O
$
and
\begin{equation}\label{res_diag}
 A_j \sim \frac{1}{2} \begin{pmatrix}
	\theta_j & \\
	& -\theta_j
	\end{pmatrix},
 \quad 
 \theta_j \notin \mathbb{Z} \quad (1\le j \le n+1).
\end{equation}
Then, the system \eqref{Eq} has regular singular points on 
$S=\{t_1,\ldots, t_n,t_{n+1}=\infty\} \subset \mathbb{P}^1$
and the local monodromy at $x=t_j$ is given by the conjugacy class of 
\begin{equation}\label{Eq_localmonodromy}
e^{2\pi \sqrt{-1}A_j} \sim \begin{pmatrix} 
 e(\theta_j) & \\
  & e(-\theta_j)
  \end{pmatrix},
\end{equation}
where $e(\alpha):=\exp{(\pi\sqrt{-1}\alpha)}$ for $\alpha \in \mathbb{C}$.
The analytic continuation of a fundamental solution matrix of \eqref{Eq} induces an $\mathrm{SL}(2,\mathbb{C})$-representation of the fundamental group $\pi_1(\mathbb{P}^1\setminus S)$.
We call it the monodromy representation. 
The image of the monodromy representation is a subgroup of the special linear group $\mathrm{SL}(2,\mathbb{C})$ and is called the monodromy group. 
The isomorphism class of the monodromy representation is called the monodromy of \eqref{Eq}. 

In this paper, we consider the unitarity of the monodromy of \eqref{Eq}, namely, we give a sufficient and necessary condition for the monodromy group to be conjugate to a subgroup of a special unitary group $\mathrm{SU}(p,q)$. 
Moreover, we actually construct such monodromies.

When $n=2$, the system \eqref{Eq} is equal to the Gauss hypergeometric equation (of SL type), which is rigid, i.e., free from accessory parameters.
In this case, the monodromy is determined only by the local monodromies \eqref{Eq_localmonodromy}.
Therefore, the unitarity condition for the monodromy has to be given by a condition for $\theta_j$'s.
In fact, it is known that the monodromy is unitary if the local exponents are all real numbers.
On the other hand, when $n\ge 3$, the system \eqref{Eq} has $2(n-2)$ accessory parameters. 
In this case, the monodromy is not determined only by the local monodromies. 
Hence, the unitarity condition is not so simple as the case $n=2$.

\medskip

Recently, the author \cite{Ada} gave a characterization and construction of irreducible unitary monodromies in the case $n=3$.
The aim of this paper is to generalize this result to the case of an arbitrary number of singularities. 
We shall explain the contents of this paper briefly.
Let $\mathcal{M}(a)^{irr}$ be the moduli space of irreducible $\mathrm{SL}(2,\mathbb{C})$-representations of $\pi_1(\mathbb{P}^1\setminus S)$ with prescribed local monodromies.
Here the parameter $a\in\mathbb{C}^{n+1}$ denotes the local monodromy data on each singular point of \eqref{Eq}.
Then the moduli space $\mathcal{M}(a)^{irr}$ can be realized as an affine algebraic set $\mathcal{S}(a) \subset \mathbb{C}^{m}$, where 
\[
m=\begin{cases} 
	3 & n=3, \\
	\binom{n}{3}+\binom{n}{2} & n\ge 4.
	\end{cases}
\]
In this paper, referring to the work of Iwasaki \cite{I2}, we introduce Zariski open subsets $\mathcal{M}^\circ(a) \subset \mathcal{M}(a)^{irr}$ and $\mathcal{S}^\circ(a)\subset \mathcal{S}(a)$ called big opens.
After that, we parametrize the monodromies in $\mathcal{M}^\circ(a)$ in terms of $\mathcal{S}^\circ(a)$
(Theorem \ref{Theorem_Parametrization}).
Using this parametrization, we show that the subset of $\mathcal{M}^\circ(a)$ consisting of unitary monodromies can be identified with $\mathcal{S}^\circ(a)\cap\mathbb{R}^{m}$.
Moreover, we determine the signatures of unitary monodromies in $\mathcal{M}^\circ(a)$ in terms of $\mathcal{S}^\circ(a)\cap\mathbb{R}^m$  (Theorem \ref{Theorem_main}).
%start modified
It is known that, when $n=3$, it holds that $\mathcal{S}^\circ(a)=\mathcal{S}(a)$ and $\mathcal{M}^\circ(a)=\mathcal{M}(a)$ for a generic $a\in\mathbb{C}^4$ (Remark \ref{Remark_n=3}).
Although these equalities are expected to hold for general $n\ge 4$, we have not obtained a proof at this time. 
On the other hand, the parametrization for monodromies belonging to the complement $\mathcal{M}(a)\setminus\mathcal{M}^\circ(a)$ is considered by the author \cite{Ada} ($n=3$ and non-generic $a\in\mathbb{C}^4$) and Calligaris-Mazzocco \cite{CM} ($n=4$).
Their expressions are some complicated and seem hard to generalize for general $n \ge 5$.
Hence, in this paper we choose to focus only on the open subsets $\mathcal{S}^\circ(a)$ and $\mathcal{M}^\circ(a)$.
%end modified

We remark that, a part of our result can be regarded as another proof of Morgan-Shalen \cite{MS} and Acosta \cite{Acosta}, which gave a criterion for monodromies in $\mathcal{M}(a)$ to be unitary in the study of character varieties.

\medskip

This paper is organized as follows.
We give a brief review of the theory of the monodromy of the Fuchsian system \eqref{Eq} and construct the moduli space $\mathcal{M}(a)^{irr}$ of irreducible monodromy representations in Section \ref{Section_Moduli}.
In Section \ref{Section_Unitary}, we shall explain the equivalence of the unitarity of monodromies and the existence of monodromy invariant Hermitian forms.
In Section \ref{Section_Parametrization}, we introduce the Zariski open subsets $\mathcal{M}^\circ(a)$ and $\mathcal{S}^\circ(a)$, and parametrize the monodromies in $\mathcal{M}^\circ(a)$ in terms of $\mathcal{S}^\circ(a)$.
In Section \ref{Section_characterization}, using the results in Sections \ref{Section_Unitary} and \ref{Section_Parametrization}, we study the unitarity of the monodromies in $\mathcal{M}^\circ(a)$.

\medskip

We end this introduction by noting that the correspondence between the obtained unitary monodromies and the accessory parameters in the system \eqref{Eq} -- Riemann-Hilbert correspondence -- is very transcendental.
It seems interesting to investigate what conditions the unitarity of the monodromy imposes on the accessory parameters of the differential equations.

%%%%%%%%%%%%%%%%
%%%%%%%%%%%%%%%%
\section{Moduli space of monodromy representations}
\label{Section_Moduli}
%%%%%%%%%%%%%%%%
%%%%%%%%%%%%%%%%

We shall give a brief review of the theory of the monodromy of the Fuchsian system \eqref{Eq}.
Take a base point $b\in \mathbb{P}^1\setminus S$ and 
$(+1)$-loops $\gamma_1,\gamma_2,\ldots,\gamma_{n+1}$ as illustrated in Figure \ref{Fig_loops}.
%%%%%%
\begin{figure}[h]
\tikzset{every picture/.style={line width=0.75pt}} %set default line width to 0.75pt        
\centering
\begin{tikzpicture}[x=0.75pt,y=0.75pt,yscale=-1,xscale=1]
%uncomment if require: \path (0,252); %set diagram left start at 0, and has height of 252

%Shape: Circle [id:dp5484160323216446] 
\draw  [fill={rgb, 255:red, 0; green, 0; blue, 0 }  ,fill opacity=1 ] (307.33,97.67) .. controls (307.33,96.19) and (308.53,95) .. (310,95) .. controls (311.47,95) and (312.67,96.19) .. (312.67,97.67) .. controls (312.67,99.14) and (311.47,100.33) .. (310,100.33) .. controls (308.53,100.33) and (307.33,99.14) .. (307.33,97.67) -- cycle ;
%Shape: Circle [id:dp16932904357436884] 
\draw  [fill={rgb, 255:red, 255; green, 255; blue, 255 }  ,fill opacity=1 ] (197.33,195.42) .. controls (197.33,193.94) and (198.53,192.75) .. (200,192.75) .. controls (201.47,192.75) and (202.67,193.94) .. (202.67,195.42) .. controls (202.67,196.89) and (201.47,198.08) .. (200,198.08) .. controls (198.53,198.08) and (197.33,196.89) .. (197.33,195.42) -- cycle ;
%Shape: Circle [id:dp7458827455109234] 
\draw  [fill={rgb, 255:red, 255; green, 255; blue, 255 }  ,fill opacity=1 ] (277.83,204.67) .. controls (277.83,203.19) and (279.03,202) .. (280.5,202) .. controls (281.97,202) and (283.17,203.19) .. (283.17,204.67) .. controls (283.17,206.14) and (281.97,207.33) .. (280.5,207.33) .. controls (279.03,207.33) and (277.83,206.14) .. (277.83,204.67) -- cycle ;
%Shape: Circle [id:dp6797365000465387] 
\draw  [fill={rgb, 255:red, 255; green, 255; blue, 255 }  ,fill opacity=1 ] (412.33,186.42) .. controls (412.33,184.94) and (413.53,183.75) .. (415,183.75) .. controls (416.47,183.75) and (417.67,184.94) .. (417.67,186.42) .. controls (417.67,187.89) and (416.47,189.08) .. (415,189.08) .. controls (413.53,189.08) and (412.33,187.89) .. (412.33,186.42) -- cycle ;
%Shape: Circle [id:dp10812379678832285] 
\draw  [fill={rgb, 255:red, 255; green, 255; blue, 255 }  ,fill opacity=1 ] (377.67,40.33) .. controls (377.67,38.86) and (378.86,37.67) .. (380.33,37.67) .. controls (381.81,37.67) and (383,38.86) .. (383,40.33) .. controls (383,41.81) and (381.81,43) .. (380.33,43) .. controls (378.86,43) and (377.67,41.81) .. (377.67,40.33) -- cycle ;
%Curve Lines [id:da5875513522931497] 
\draw    (310,97.67) .. controls (266.55,151.58) and (210.67,220.33) .. (186.88,214.25) .. controls (163.08,208.17) and (213.75,108.5) .. (310,97.67) -- cycle ;
%Curve Lines [id:da8639750563288829] 
\draw    (310,97.67) .. controls (314.25,102.5) and (306.67,243.67) .. (274.67,225.67) .. controls (242.67,207.67) and (299.63,114) .. (310,97.67) -- cycle ;
%Curve Lines [id:da21409401699777164] 
\draw    (310.33,97.67) .. controls (312.33,100) and (478.67,166.83) .. (439.67,204.33) .. controls (400.67,241.83) and (310.58,94.42) .. (310.33,97.67) -- cycle ;
%Curve Lines [id:da8630698840746049] 
\draw    (310,97.67) .. controls (429.75,81) and (445.75,46.5) .. (402.75,11.5) .. controls (359.75,-23.5) and (327,85.17) .. (310,97.67) -- cycle ;
\draw  [fill={rgb, 255:red, 0; green, 0; blue, 0 }  ,fill opacity=1 ] (179.85,186.2) -- (181.25,198.32) -- (189.72,189.54) -- (183.02,193.1) -- cycle ;
\draw  [fill={rgb, 255:red, 0; green, 0; blue, 0 }  ,fill opacity=1 ] (259.24,204.11) -- (265.8,214.4) -- (269.58,202.81) -- (265.1,208.93) -- cycle ;
\draw  [fill={rgb, 255:red, 0; green, 0; blue, 0 }  ,fill opacity=1 ] (389.25,200.17) -- (401.24,202.41) -- (395.36,191.73) -- (396.77,199.18) -- cycle ;
\draw  [fill={rgb, 255:red, 0; green, 0; blue, 0 }  ,fill opacity=1 ] (419.69,20.64) -- (408.25,16.41) -- (412.24,27.93) -- (412.1,20.35) -- cycle ;
%Shape: Circle [id:dp10660813041049666] 
\draw  [fill={rgb, 255:red, 0; green, 0; blue, 0 }  ,fill opacity=1 ] (313.42,209.38) .. controls (313.42,208.89) and (313.81,208.5) .. (314.29,208.5) .. controls (314.77,208.5) and (315.17,208.89) .. (315.17,209.38) .. controls (315.17,209.86) and (314.77,210.25) .. (314.29,210.25) .. controls (313.81,210.25) and (313.42,209.86) .. (313.42,209.38) -- cycle ;
%Shape: Circle [id:dp3092539897097275] 
\draw  [fill={rgb, 255:red, 0; green, 0; blue, 0 }  ,fill opacity=1 ] (332.42,210.13) .. controls (332.42,209.64) and (332.81,209.25) .. (333.29,209.25) .. controls (333.77,209.25) and (334.17,209.64) .. (334.17,210.13) .. controls (334.17,210.61) and (333.77,211) .. (333.29,211) .. controls (332.81,211) and (332.42,210.61) .. (332.42,210.13) -- cycle ;
%Shape: Circle [id:dp6693303409019935] 
\draw  [fill={rgb, 255:red, 0; green, 0; blue, 0 }  ,fill opacity=1 ] (352.17,207.88) .. controls (352.17,207.39) and (352.56,207) .. (353.04,207) .. controls (353.52,207) and (353.92,207.39) .. (353.92,207.88) .. controls (353.92,208.36) and (353.52,208.75) .. (353.04,208.75) .. controls (352.56,208.75) and (352.17,208.36) .. (352.17,207.88) -- cycle ;
%Shape: Circle [id:dp28833274837809664] 
\draw  [fill={rgb, 255:red, 0; green, 0; blue, 0 }  ,fill opacity=1 ] (370.67,203.13) .. controls (370.67,202.64) and (371.06,202.25) .. (371.54,202.25) .. controls (372.02,202.25) and (372.42,202.64) .. (372.42,203.13) .. controls (372.42,203.61) and (372.02,204) .. (371.54,204) .. controls (371.06,204) and (370.67,203.61) .. (370.67,203.13) -- cycle ;

% Text Node
\draw (199.5,171.73) node [anchor=north west][inner sep=0.75pt]    {$t_{1}$};
% Text Node
\draw (276.5,181.32) node [anchor=north west][inner sep=0.75pt]    {$t_{2}$};
% Text Node
\draw (402.17,166.07) node [anchor=north west][inner sep=0.75pt]    {$t_{n}$};
% Text Node
\draw (345,48.57) node [anchor=north west][inner sep=0.75pt]    {$t_{n+1} =\infty $};
% Text Node
\draw (302.33,75.4) node [anchor=north west][inner sep=0.75pt]    {$b$};
% Text Node
\draw (181.67,220.82) node [anchor=north west][inner sep=0.75pt]    {$\gamma _{1}$};
% Text Node
\draw (270.83,231.48) node [anchor=north west][inner sep=0.75pt]    {$\gamma _{2}$};
% Text Node
\draw (430.42,214.9) node [anchor=north west][inner sep=0.75pt]    {$\gamma _{n}$};
% Text Node
\draw (431.5,26.9) node [anchor=north west][inner sep=0.75pt]    {$\gamma _{n+1}$};
\end{tikzpicture}
\caption{$(+1)$-loops}
\label{Fig_loops}
\end{figure}
%%%%%%%%%%%%%%
Then, the fundamental group $\pi_1(\mathbb{P}^1 \setminus S,b)$ has the presentation
\[
 \pi_1(\mathbb{P}^1 \setminus S,b)=\langle \gamma_1,\gamma_2,\ldots,\gamma_{n+1}\, \vert\,
 \gamma_1\gamma_2\cdots \gamma_{n+1}=1 \rangle,
\]
where we identify the $(+1)$-loop with its equivalence class.
We take a fundamental solution matrix $\mathcal{U}(x)$ of \eqref{Eq} in a neighborhood of $x=b$ and let $\gamma_*\,\mathcal{U}(x)$ be the analytic continuation of $\mathcal{U}(x)$ along a loop $\gamma \in \pi_1(\mathbb{P}^1 \setminus S,b)$.
Then there exists a matrix $M_j \in \mathrm{SL}(2,\mathbb{C})$ such that
\[
(\gamma_{j})_*\,\mathcal{U}(x)=\mathcal{U}(x) M_j.
\]
We note that 
\[
M_j \sim \begin{pmatrix} 
 e(\theta_j) & \\
  & e(-\theta_j)
  \end{pmatrix}
\]
holds from the assumption \eqref{res_diag}.
By setting $\rho(\gamma_j)=M_j$, we obtain the (anti-)representation 
\[
\rho:\pi(\mathbb{P}^1 \setminus S,b)\,\to\,\mathrm{SL}(2,\mathbb{C}).
\]
We call this representation the monodromy representation of $\mathcal{U}(x)$.
The image of the monodromy representation is called the monodromy group of $\mathcal{U}(x)$.
Here we note that, to determine the monodromy representation, it is sufficient to fix the image $M_j$ of each generator $\gamma_j$ of $\pi_1(\mathbb{P}^1\setminus S,b)$.
The relation $\gamma_1\gamma_2\cdots\gamma_{n+1}=1$ leads to 
\begin{equation}
M_{n+1}M_n\cdots M_1=I. \label{Mn_prod}
\end{equation}
Thanks to this relation, the matrix $M_{n+1}$ is determined by $M_1,\ldots,M_n$.
Hence, we sometimes identify the representation $\rho$ with the tuple of matrices $(M_1,\ldots,M_{n})$, and write $\rho=(M_1,\ldots,M_n)$.

The monodromy representation $\rho=(M_1,\ldots,M_n)$ depends on the fundamental solution matrix $\mathcal{U}(x)$.
If we consider the transformation of the fundamental solution matrix 
\[
\mathcal{U}(x) \mapsto \mathcal{V}(x)P 
\]
by some $P \in \mathrm{SL}(2,\mathbb{C})$, then the monodromy representation of $\mathcal{V}(x)$ is given by
\[
P\rho P^{-1}=(PM_1 P^{-1},PM_2P^{-1},\ldots,PM_nP^{-1}).
\]
This means that a change of the fundamental solution induces the isomorphism of monodromy representations.
We call the isomorphism class 
\[
[\rho]:=\{P\rho P^{-1}\,;\, P\in\mathrm{SL}(2,\mathbb{C})\}
\]
the monodromy of \eqref{Eq}. 
A monodromy representation $\rho=(M_1,\ldots,M_n)$ or an isomorphism class $[\rho]=[(M_1,\ldots,M_{n})]$ is called irreducible if there is no common invariant subspace of $(M_1,\ldots,M_n)$ except $\{0\}$ and $\mathbb{C}^2$.
The Riemann-Hilbert correspondence guarantees that, 
for any given irreducible monodromy $[\rho]$, there exists a Fuchsian system \eqref{Eq} whose monodromy is $[\rho]$.
%%%%%%%

Let us define the moduli space of monodromy representations of Fuchsian systems \eqref{Eq}.
We set $a=(a_1,a_2,\ldots,a_{n+1}) \in \mathbb{C}^{n+1}$ by
\begin{equation}\label{Eq_a_j}
 a_j=\tr e^{2\pi \sqrt{-1}A_j}=e(\theta_j)+e(-\theta_j)=2\cos \pi \theta_j.
\end{equation}
Since we assumed $\theta_j \notin \mathbb{Z}$ in \eqref{res_diag}, we have
\begin{equation}\label{Tr_ass}
a_j \neq \pm 2 \quad (1\le j \le n+1).
\end{equation}
Next we set
\[
 \mathcal{O}(a_j)=\{M \in \mathrm{SL}(2,\mathbb{C})\,;\, \tr (M)=a_j\}
\]
for $j=1,2,\ldots,n+1$.
Thanks to \eqref{Tr_ass}, any matrix in $\mathcal{O}(a_j)$ has distinct eigenvalues $\{e(\theta_j), e(-\theta_j)\}$ and hence it is diagonalizable.
In other words, the set $\mathcal{O}(a_j)$ can be regarded as the conjugacy class of $e^{2\pi \sqrt{-1}A_j}$, namely,
\[
\mathcal{O}(a_j)=[e^{2\pi \sqrt{-1}A_j}]=\{Pe^{2\pi \sqrt{-1}A_j}P^{-1}\,;\,P \in \mathrm{SL}(2,\mathbb{C})\}.
\]
Therefore, the set $\mathcal{O}(a_j)$ is nothing but the local monodromy at $x=t_j$ under the condition \eqref{Tr_ass}.
%representation variety
Now, we define the moduli space by
\[
\mathcal{M}(a):=\{(M_1,\ldots,M_n) \in \mathcal{O}(a_1)\times 
  \cdots \times \mathcal{O}(a_n)
 \,;\,  M_{n+1}\in \mathcal{O}(a_{n+1})\}/\sim,
\]
where the equivalence relation $\rho' \sim \rho$ means there exists
$P \in \mathrm{SL}(2,\mathbb{C})$ such that $\rho'=P\rho P^{-1}$ holds.
We denote by $\mathcal{M}(a)^{irr}$ the subset of $\mathcal{M}(a)$ which consists of irreducible elements.
It is known that the moduli space $\mathcal{M}(a)^{irr}$ is a $2(n-2)$ dimensional algebraic variety.
Especially, if $n=2$, the space $\mathcal{M}(a)^{irr}$ consists of a point.
This means that the equation \eqref{Eq} is rigid when $n=2$.
%Before that, we introduce the notion of unitary monodromies, which is a main theme of this paper.
\begin{remark}
The condition \eqref{Tr_ass} comes from the assumption \eqref{res_diag} for the differential equation.
On the other hand, the set $\mathcal{O}(a_j)$ and the moduli space $\mathcal{M}(a)$ themselves can be defined without assuming \eqref{Tr_ass}.
Furthermore, all the results in the following sections hold without assuming \eqref{Tr_ass}.
\end{remark}

%%%%%%%%%%%%%%%%
%%%%%%%%%%%%%%%%
\section{Unitary monodromies and invariant Hermitian forms}
\label{Section_Unitary}
%%%%%%%%%%%%%%%%
%%%%%%%%%%%%%%%%

Let $p,q$ be non-negative integers satisfying $p+q=2$ and 
$I_{p,q}=I_p\oplus (-I_q)$.
Here $I_n$ denotes the identity matrix of size $n$.
Then the special unitary group $\mathrm{SU}(p,q)$ with the signature $(p,q)$ is defined by
\[
\mathrm{SU}(p,q)=\{ U \in \mathrm{SL}(2,\mathbb{C})\,;\,
\bar{U}^T I_{p,q} U=I_{p,q}\},
\]
where $\bar{U}^T$ denotes the complex conjugate of the transpose of $U$.
We sometimes write $\mathrm{SU}(2,0)=\mathrm{SU}(0,2)=\mathrm{SU}(2)$.
%%%%
\begin{definition}
A monodromy $[\rho]\in\mathcal{M}(a)^{irr}$ is called \textit{unitary}
(of the signature $(p,q)$)
 if there is a representative 
$\rho=(M_1,\ldots,M_n)$ 
such that the monodromy group 
$\mathrm{Im\,}\rho=\langle M_1,\ldots,M_n\rangle$
is a subgroup of the special unitary group $\mathrm{SU}(p,q)$.
\end{definition}
%%%%
%%%%
The unitarity of the monodromy $[\rho]$ is equivalent to the existence of an invariant non-degenerate Hermitian form.
%Invariant Hermitian forms
%%%%
\begin{proposition}\label{Proposition_rep_Herm}
The monodromy $[\rho]\in\mathcal{M}(a)^{irr}$ is unitary if and only if there is a representative $\rho=(M_1,\ldots,M_n)$ of $[\rho]$ which has an invariant non-degenerate Hermitian form $h$,
namely, there is a non-degenerate Hermitian matrix $H$ satisfying
\begin{equation}\label{eq_invHerm}
   \bar{M}_j^T HM_j=H \quad (j=1,2,\ldots,n).
 \end{equation}
If such an invariant Hermitian matrix $H$ is definite (resp. indefinite), then the signature of $[\rho]$ is $(2,0)$ or $(0,2)$ (resp. $(1,1)$).
\end{proposition}
%%%%
\begin{proof}
Assume that the monodromy $[\rho]$ is unitary of the signature $(p,q)$.
Then, there exists a representative $\rho=(M_1,\ldots,M_n)$ such that
\[
\mathrm{Im}\,\rho=\langle M_1,\ldots,M_n\rangle \subset \mathrm{SU}(p,q).
\]
By setting $H=I_{p,q}$, we have \eqref{eq_invHerm}.

Conversely, we assume that the representative $\rho=(M_1,\ldots,M_n)$ of $[\rho]$ has an invariant non-degenerate Hermitian matrix $H$ satisfying \eqref{eq_invHerm}.
Multiplying the both side of \eqref{eq_invHerm} by a suitable positive real number, we can suppose that $|\det H|=1$.
Here we recall an elementary fact: any Hermitian matrix can be diagonalized by a unitary matrix, and the resulting diagonal matrix has only real entries.
Hence we can take a unitary matrix $U\in\mathrm{SU}(2)$ such that
\[
U^{-1}HU=\bar{U}^THU=\begin{pmatrix}
h_1 & \\
 & h_2
\end{pmatrix},
\]
where $h_1,h_2 \in\mathbb{R}\setminus\{0\}$ are the eigenvalues of $H$.
Note that $|h_1 h_2|=1$.
Without loss of generality, we assume $h_1 \ge h_2$.
Next we set
\[
Q=\begin{pmatrix}
\sqrt{|h_1|} & \\
 & \sqrt{|h_2|}
\end{pmatrix}
\]
and $P=QU^{-1}$. 
Thanks to $|h_1h_2|=1$, we have $P \in \mathrm{SL}(2,\mathbb{C})$.
Then $\rho_1:=P\rho P^{-1}=(PM_1P^{-1},\ldots,PM_nP^{-1})$ 
is also a representative of $[\rho]$ and keep the Hermitian matrix $(\bar{P}^T)^{-1} H P^{-1}$ invariant. 
Here we have
\[
(\bar{P}^T)^{-1} H P^{-1}=(\bar{Q}^T)^{-1} \bar{U}^THUQ^{-1}
 =\begin{pmatrix}
 \dfrac{h_1}{|h_1|} & \\
 & \dfrac{h_2}{|h_2|}
 \end{pmatrix}=I_{p,q}.
\]
This means that $\mathrm{Im}\,\rho_1$ is a subgroup of $\mathrm{SU}(p,q)$, that is, the monodromy $[\rho]$ is unitary.
 If $H$ is definite, namely $h_1h_2>0$, then we have $(p,q)=(2,0)$ or $(0,2)$.
On the other hand, if $H$ is indefinite, namely $h_1h_2<0$, then we have $(p,q)=(1,1)$.
\end{proof}
%%%
%%%%%
\begin{remark}\label{Rem_invHerm}
If a representative $\rho_1\in[\rho]$
has an invariant non-degenerate Hermitian form $H_1$, then any other representative
$\rho_2=P\rho_1 P^{-1}\in[\rho]$ 
also has an invariant non-degenerate Hermitian form $H_2=(\bar{P}^{-1})^TH_1P^{-1}$.
Moreover, the signatures of $H_1$ and $H_2$ are the same.
As a consequence, 
we see that if a monodromy $[\rho]$ is unitary, then in fact any representative has an invariant non-degenerate Hermitian form.
\end{remark}
%%%%%
%Therefore, Proposition \ref{Proposition_rep_Herm} can be rephrased as follows.
Proposition \ref{Proposition_rep_Herm} will play an important role for the characterization of unitary monodromies in Section \ref{Section_characterization}.

%%%%%%%%%%%%%%%%
%%%%%%%%%%%%%%%%
\section{Parametrization theorem}
\label{Section_Parametrization}
%%%%%%%%%%%%%%%%
%%%%%%%%%%%%%%%%

Hereafter, we assume $n \ge 3$ for simplicity.
But we remark that the case $n=2$ can be considered in the same idea.
In this section, we realize the moduli space $\mathcal{M}(a)^{irr}$ as an affine algebraic set, and parametrize the elements in $\mathcal{M}(a)^{irr}$ in terms of this affine algebraic set.

For each $[\rho]=[(M_1,\ldots,M_n)]\in\mathcal{M}(a)^{irr}$,
we set
 \begin{equation}\label{tr_cood}
 x_{ji}=\tr (M_jM_i), \quad x_{kji}=\tr(M_kM_jM_i), \quad \{i,j,k\}\subset\{1,2,\ldots,n\}.
 \end{equation}
Here we note that the values of $x_{ji}$ and $x_{kji}$ do not depend on the choice of the representative of $[\rho]$,
and satisfy some linear relations as follows.
\begin{lemma}\label{Lemma_tr}
For any $\{i,j,k\} \subset \{1,2,\ldots,n\}$, it hold that
	\begin{enumerate}\rm
	\item $x_{ji}=x_{ij}$,
	\item $x_{kji}=x_{jik}=x_{ikj}$,
	\item $x_{kji}=a_k x_{ji}+a_j x_{ki}+a_i x_{kj}-a_k a_j a_i -x_{kij}$.
	\end{enumerate}
\end{lemma}
%%%%
\begin{proof}
The statements (i) and (ii) are clear from the definition \eqref{tr_cood}.
The last statement (iii) can be shown by the iterations of the skein relation
\[
\tr (AB)+\tr(AB^{-1})=\tr (A) \tr (B) \quad (A,B \in\mathrm{SL}(2,\mathbb{C}))
\]
for $x_{kji}=\tr(M_kM_jM_i)$.
\end{proof}
%%%%
Thanks to this lemma, to obtain the full set of \eqref{tr_cood}, it is sufficient to give only 
 $x_{ji}$ for $1\le i<j \le n$ and 
 $x_{kji}$ for $1\le i<j<k\le n$.
Using them, we set the coordinates $x \in \mathbb{C}^m$ as 
\begin{equation}\label{tr_cood_2}
 x=
\begin{cases}
 (x_{21},x_{31},x_{32}) & n=3, \\
 (x_{21},x_{31},x_{32},\ldots,x_{n,n-1},x_{321},\ldots,x_{n,n-1,n-2}) & n\ge4,
 \end{cases}
\end{equation} 
where 
\[
m=\begin{cases} 
	3 & n=3, \\
	\binom{n}{3}+\binom{n}{2} & n\ge 4.
	\end{cases}
\]
Moreover, there exist non-trivial relations between the entries of $x$.

\begin{proposition}[cf. Ashley-Burelle-Lawton \cite{ABL}, Theorem 3.1]\label{Prop_invariants}
Let $x \in \mathbb{C}^m$ be the coordinates \eqref{tr_cood_2} of $[\rho]\in\mathcal{M}(a)^{irr}$.
Then there exist polynomials 
$ f_1(x),f_2(x),\ldots,f_{r}(x) \in \mathbb{C}[x]$ such that
\[
f_1(x)=f_2(x)=\cdots=f_{r}(x)=0,
\]
where 
\[ 
 r=
\begin{cases}
1 & n=3, \\[5pt]
\dfrac{1}{2}\left(\binom{n}{3}^2+\binom{n}{3}\right)+n\binom{n}{4}+1 & n\ge 4.
\end{cases}
\]
The polynomials $f_1(x),f_2(x),\ldots,f_{r}(x)$ are described as follows:
for $1\le i_1<i_2<i_3\le n$, we set 
\[
\begin{aligned}
s_3(i_1,i_2,i_3) =
a_{i_1}x_{i_3i_2}+a_{i_2}x_{i_3i_1}+a_{i_3}x_{i_2i_1}-a_{i_3}a_{i_2}a_{i_1}-2 x_{i_3 i_2 i_1}.
\end{aligned}
\]

 {\bf Type 1 relations:} when $n \ge 3$, the $\frac{1}{2}\left(\binom{n}{3}^2+\binom{n}{3}\right)$ polynomials $f_1(x),\ldots,f_{\frac{1}{2}\left(\binom{n}{3}^2+\binom{n}{3}\right)}(x)$ are defined by
 \[
 s_3(i_1,i_2,i_3)s_3(j_1,j_2,j_3)+2\det Z,
 \]
 where $1\le i_1<i_2<i_3\le n$, $1\le j_1<j_2<j_3\le n$ and 
 \[
 Z=\begin{pmatrix}
 z_{i_1,j_1} & z_{i_1,j_2} & z_{i_1,j_3} \\
 z_{i_2,j_1} & z_{i_2,j_2} & z_{i_2,j_3} \\
 z_{i_3,j_1} & z_{i_3,j_2} & z_{i_3,j_3}
 \end{pmatrix},
 \quad 
 z_{i,j}=\begin{cases} 
  \dfrac{1}{2}a_{i}^2-2 & i=j \\
 x_{ij}-\dfrac{1}{2}a_i a_j & i\neq j.
 \end{cases}
 \]
 
\medskip
 
  {\bf Type 2 relations:} when $n \ge 4$, the $n\binom{n}{4}$ polynomials $f_{\frac{1}{2}\left(\binom{n}{3}^2+\binom{n}{3}\right)+1}(x),\ldots,f_{r-1}(x)$ are defined by
  \[
   \sum_{k=0}^3(-1)^k z_{i,p_k} s_3(p_0,\ldots,\widehat{p_k},\ldots,p_3),
  \]
  where $1\le i \le n$, $1\le p_0<p_1<p_2<p_3\le n$ and $\widehat{p_k}$ means omission.

\medskip

  {\bf Type 3 relation:} when $n \ge 4$, the polynomial $f_r(x)$ is defined as follows. 
  Firstly, we define the polynomials
  \[
  g_{i_k,i_{k-1},\ldots,i_1}(x) \in \mathbb{C}[x] \quad (1\le i_1<\cdots<i_k\le n, ~ 1\le k\le n)
  \]
 recursively: for each $k=1,2,3$ we set
  \[
  \begin{aligned}
  g_{i_1}(x)=a_{i_1} \quad &(1\le i_1\le n), \\
  g_{i_2,i_1}(x)=x_{i_2i_1} \quad &(1\le i_1<i_2\le n),\\
  g_{i_3,i_2,i_1}(x)=x_{i_3i_2i_1} \quad &(1\le i_1<i_2<i_3\le n).
  \end{aligned}
  \]
 Using them, we define $g_{i_k,i_{k-1},\ldots,i_1}(x)$ for $k\ge4$ by the following recurrence relation
  \begin{equation}\label{Poly_g}
  \begin{aligned}
  g_{i_k,i_{k-1},\ldots,i_1}(x)=&\frac{1}{2} \{g_{i_k,\ldots,i_4}(x)a_{i_3}a_{i_2}a_{i_1}+g_{i_k,\ldots,i_4}(x)
  x_{i_3i_2i_1}
  +\sum_{p=1}^{3} a_{i_p}g_{i_k,\ldots,\widehat{i_p},\ldots,i_1}(x) \\
  &+g_{i_k,\ldots,i_3}(x)x_{i_2i_1}-g_{i_k,\ldots,i_4,i_2}(x)x_{i_3i_1}+g_{i_k,\ldots,i_4,i_1}(x)x_{i_3i_2}  \\
 & -g_{i_k,\ldots,i_4}(x)a_{i_3}x_{i_2i_1}-g_{i_k,\ldots,i_4}(x)a_{i_1}x_{i_3i_2} 
  -g_{i_k,\ldots,i_4,i_1}(x)a_{i_3}a_{i_2}-g_{i_k,\ldots,i_3}(x)a_{i_2}a_{i_1}\},
    \end{aligned}
  \end{equation}
 where $\widehat{i_p}$ means omission. 
 Then, the desired polynomial $f_r(x)$ is defined by
 \[
 f_{r}(x)=g_{n,n-1,\ldots,1}(x)-a_{n+1}.
 \]
\end{proposition}

%%%%
\begin{remark}\label{remark_type3}
In \cite{ABL}, Ashley-Burelle-Lawton considered 
$\mathrm{Hom}(F_n,\mathrm{SL}(2,\mathbb{C}))/\!/ \mathrm{SL}(2,\mathbb{C})$ and derived polynomials of Types 1 and 2, where $F_n$ is a free group of rank $n$.
When $n=3$, we have $x_{321}=a_4$ and the polynomial $f_1(x)$ is nothing but the cubic polynomial considered by many authors (Jimbo \cite{J}, Iwasaki \cite{I2}, etc.) in the study of Painlev\'{e} differential equation.
When $n\ge 4$, the set 
\[
\{x \in \mathbb{C}^m\,;\,f_1(x)=\cdots=f_{r-1}(x)=0\}
\]
corresponds to the moduli space
\begin{equation}\label{Eq_InM}
\{(M_1,\ldots,M_n) \in \mathcal{O}(a_1)\times 
  \cdots \times \mathcal{O}(a_n)\}/\sim.
\end{equation}
Therefore
we have to add the condition $f_r(x)=0$
in order to add the condition $M_{n+1}\in \mathcal{O}(a_{n+1})$ to \eqref{Eq_InM}. 
The fact that the condition $f_r(x)=0$ is equal to the condition $M_{n+1} \in \mathcal{O}(a_{n+1})$ can be shown by using Lemma \ref{Lemma_reduction} stated later. 
We will give a more detailed explanation in Appendix \ref{Appendix}.
\end{remark}
%%%%

Using the polynomials in Proposition \ref{Prop_invariants}, we set the affine algebraic set
\[
\mathcal{S}(a)=\{x \in \mathbb{C}^m\,;\,f_1(x)=\cdots=f_{r}(x)=0\}.
\]
Then the assignment of the coordinates \eqref{tr_cood_2} defines the mapping 
\begin{equation}\label{Eq_Phi}
 \begin{array}{cccc}
\Phi:& \mathcal{M}(a)^{irr} & \to & \mathcal{S}(a)  \\
       &\rotatebox{90}{$\in$} & & \rotatebox{90}{$\in$} \\
       & [\rho] & \mapsto & x
\end{array}.
\end{equation}
%%%%
\begin{proposition}[cf. Lubotzky-Magid \cite{LM}, Theorem 1.28]
The map \eqref{Eq_Phi} is injective.
\end{proposition}
%%%%

In the following, referring to Iwasaki's work \cite{I2}, we parametrize the elements in $\mathcal{M}(a)^{irr}$ in terms of $\mathcal{S}(a)$, that is, we construct the ``inverse" mapping of $\Phi$. 
But as pointed out by Iwasaki, it seems difficult to construct the ``global" inverse mapping of $\Phi$. 
So we try to construct ``local" inverse mappings on some open subsets of $\mathcal{M}(a)^{irr}$ and $\mathcal{S}(a)$, and glue them to construct the desired inverse mapping.
To explain this, we prepare some notions.

For any polynomial $p=p(x) \in \mathbb{C}[x]$, we define the Zariski open subsets
\[
\begin{aligned}
\mathcal{S}(a)[p]&=\{x\in\mathcal{S}(a)\,;\, p(x) \neq 0\}, \\
\mathcal{M}(a)[p]&=\{[\rho] \in \mathcal{M}(a)^{irr}\,;\,p(\Phi([\rho]))\neq 0\}
\end{aligned}
\]
of $\mathcal{S}(a)$ and $\mathcal{M}(a)^{irr}$, respectively.
Next we define the subset $\mathcal{I}_n \subset \mathbb{N}^2$ by
 \begin{equation}\label{Eq_I_n}
 \mathcal{I}_n=\left\{(j,k) \in \mathbb{N}^2\,;\,
 1\le j \le n~\text{and $k\in\mathbb{N}$ is given by $(*)$ for each $j$}
 \right\},
 \end{equation}
 where the condition $(*)$ is 
 \[
  \left\{
  \begin{array}{ccc}
 j=1 & \Rightarrow & 2\le k<n, \\
 2\le j<n & \Rightarrow & j<k\le n, \\
 j=n & \Rightarrow & k=1.
 \end{array}
 \right.
 \]
 For example, when $n=3$, the set $\mathcal{I}_3$ is given by
 $
 \{(1,2),(2,3),(3,1)\}.
 $
 
For each $(j,k)\in \mathcal{I}_n$, we set
\begin{equation} \label{Eq_p_kj}
p_{kj}^{(i_0)}(x)= \begin{cases}
  (x_{kj}^2-4)\psi(x_{kj},a_{k},a_{j}) & i_0=0,\\
  (x_{kj}^2-4)\psi(x_{kji_0},x_{kj},a_{i_0}) & i_0 \in \{1,2,\ldots,n\}\setminus\{j,k\},
  \end{cases}
\end{equation}
where the polynomial $\psi(s,t,u)$ is defined by
\begin{equation}\label{Eq_psi}
 \psi(s,t,u)=s^2+t^2+u^2-stu-4.
\end{equation}
\begin{remark}
Note that $i_0\in\{1,2,\ldots,n\}\setminus\{j,k\}$ in \eqref{Eq_p_kj} does not necessarily satisfy $i_0<j<k$.   
In the case $j<i_0<k$ (resp. $j<k<i_0$), we understand that $x_{kji_0}$ means the polynomial of $x_{ki_0j}$ (resp. $x_{i_0kj}$) given by the relation in Lemma \ref{Lemma_tr} (iii) (resp. (ii)).
Also, in the case $k=1<i_0<n=j$, we understand that $x_{kji_0}=x_{1ni_0}=x_{ni_01}$ by the relation in Lemma \ref{Lemma_tr} (ii).
Hereafter, we shall adopt such understanding as appropriate.
\end{remark}
Let us define the open sets
\begin{equation}
 \mathcal{S}_{kj}^{(i_0)}(a)=\mathcal{S}(a)[p_{kj}^{(i_0)}], \quad 
 \mathcal{M}_{kj}^{(i_0)}(a)=\mathcal{M}(a)[p_{kj}^{(i_0)}].
\end{equation}
Then we see $\Phi(\mathcal{M}_{kj}^{(i_0)}(a)) \subset \mathcal{S}_{kj}^{(i_0)}(a)$ immediately.
Hence the mapping
\begin{equation}\label{Eq_Phi_rest}
\Phi_{kj}^{(i_0)}:=\Phi \vert_{\mathcal{M}_{kj}^{(i_0)}}\,:\,\mathcal{M}_{kj}^{(i_0)}(a) ~  \to ~ \mathcal{S}_{kj}^{(i_0)}(a)
\end{equation}
is well-defined.
Our aim is to construct the inverse mapping of $\Phi_{kj}^{(i_0)}$.
Let us fix a square root of $x_{kj}^2-4$ and put
\begin{equation}\label{r_kj}
r_{kj}=\sqrt{x_{kj}^2-4},\quad \lambda_{kj}^{\pm}=\frac{x_{kj}\pm r_{kj}}{2}.
\end{equation}
Now we define the following.

%%%%
%%%%
\begin{definition}\label{Definition_NormalForm}
For each $(j,k)\in\mathcal{I}_n$ and $i_0 \in \{0,1,\ldots,n\}\setminus\{j,k\}$, let 
\begin{equation}\label{Eq_varphi_kj}
\varphi_{kj}^{(i_0)}: \mathcal{S}_{kj}^{(i_0)}(a)~\to~\mathcal{M}_{kj}^{(i_0)}(a)
\end{equation}
be the map associating each $x \in \mathcal{S}_{kj}^{(i_0)}(a)$ with the conjugacy class $[\rho]=[(M_1,\ldots,M_n)] \in \mathcal{M}_{kj}^{(i_0)}(a)$ whose representative 
$\rho=(M_1,\ldots,M_n)$ is given as follows;
 %%%%%%%%%%%%
 %%%%%%%%%%%%
 	\begin{enumerate}
	\item In the case $i_0=0$,
		  we take a representative $\rho=(M_1,\ldots,M_n)$ of the form
\begin{equation}\label{Eq_rep_case1}
\begin{aligned}
		&M_j=\begin{pmatrix}
			 -\dfrac{a_k-\lambda_{kj}^+ a_j}{r_{kj}} & -\dfrac{\psi(x_{kj},a_k,a_j)}{r_{kj}^2} \\[11pt]
			 1 & \dfrac{a_k-\lambda_{kj}^- a_j}{r_{kj}}
			 \end{pmatrix}, \\
		&M_k=\begin{pmatrix}
			 -\dfrac{a_j-\lambda_{kj}^+ a_k}{r_{kj}} & \dfrac{\lambda_{kj}^+ \psi(x_{kj},a_k,a_j)}{r_{kj}^2} \\[11pt]
			 -\lambda_{kj}^- & \dfrac{a_j-\lambda_{kj}^- a_k}{r_{kj}}
			 \end{pmatrix},
			 \\
		&M_i=\begin{pmatrix}
			 u_{11}^{(i)} & u_{12}^{(i)} \\
			 u_{21}^{(i)} & u_{22}^{(i)}
 			 \end{pmatrix} \quad (i \in \{1,2,\ldots,n\}\setminus\{j,k\}),
\end{aligned}
\end{equation}
where the diagonal entries of $M_i$ are given by
\begin{equation}\label{Eq_M_i_diag}
	\begin{aligned}
	u_{11}^{(i)}=\frac{x_{kji}-\lambda_{kj}^- a_i}{r_{kj}}, \quad
	u_{22}^{(i)}=-\frac{x_{kji}-\lambda_{kj}^+ a_i}{r_{kj}}
	\end{aligned}
\end{equation}
and the anti-diagonal entries are given by
\begin{equation}\label{Eq_u12_u21_i_case1}
	\begin{aligned}
	u_{12}^{(i)}&=\frac{x_{ki}-a_k u_{22}^{(i)}+\lambda_{kj}^+(x_{ji}-a_j u_{11}^{(i)})}{r_{kj}}, \\
	u_{21}^{(i)}&=\frac{r_{kj}\{x_{ki}-a_k u_{11}^{(i)}+\lambda_{kj}^-(x_{ji}-a_j u_{22}^{(i)})\}}{\psi(x_{kj},a_k,a_j)}.
	\end{aligned}
\end{equation}
%%%%
	\item In the case 
	$i_0 \in \{1,2,\ldots,n\}\setminus\{j,k\}$,
	we take a representative $\rho=(M_1,\ldots,M_n)$ of the form
\begin{equation}\label{Eq_rep_case2}
\begin{aligned}
		&M_j=\begin{pmatrix}
			 -\dfrac{a_k-\lambda_{kj}^+ a_j}{r_{kj}} & u_{12}^{(j)} \\[11pt]
			 u_{21}^{(j)} & \dfrac{a_k-\lambda_{kj}^- a_j}{r_{kj}}
			 \end{pmatrix},
			 \quad 
		M_k=\begin{pmatrix}
			 -\dfrac{a_j-\lambda_{kj}^+ a_k}{r_{kj}} & -\lambda_{kj}^+ u_{12}^{(j)} \\[11pt]
			 -\lambda_{kj}^- u_{21}^{(j)} & \dfrac{a_j-\lambda_{kj}^- a_k}{r_{kj}}
			 \end{pmatrix},
			 \\
			 &
	M_{i_0}=\begin{pmatrix}
			u_{11}^{(i_0)} & -\dfrac{\psi(x_{kji_0},x_{kj},a_{i_0})}{r_{kj}^2}\\[11pt]
	 		 1 & u_{22}^{(i_0)}
			 \end{pmatrix}, \\
	&M_i=\begin{pmatrix}
			 u_{11}^{(i)} & u_{12}^{(i)} \\
			 u_{21}^{(i)} & u_{22}^{(i)}
 			 \end{pmatrix} \quad (i \in \{1,2,\ldots,n\}\setminus\{j,k,i_0\}),
		\end{aligned}
\end{equation}
where the diagonal entries of $M_{i_0}$ and $M_i$ are given by
\begin{equation}\label{Eq_u11u22_i_case2}
	\begin{aligned}
	u_{11}^{(i)}= \dfrac{x_{kj{i}}-\lambda_{kj}^- a_{i}}{r_{kj}}, \quad
	u_{22}^{(i)}=-\dfrac{x_{kj{i}}-\lambda_{kj}^+ a_{i}}{r_{kj}}, 
	\quad (i \in \{1,2,\ldots,i_0,\ldots,n\}\setminus\{j,k\})
	\end{aligned}
\end{equation}
and the anti-diagonal entries of $M_j,M_k$ and $M_i$ are given by
\begin{equation}\label{Eq_u12u21_ij_case2}
	\begin{aligned}
	u_{12}^{(j)}&=-\frac{x_{k{i_0}}-a_k u_{11}^{(i_0)}+\lambda_{kj}^-(x_{j{i_0}}-a_j u_{22}^{(i_0)})}{r_{kj}}, \\
	u_{21}^{(j)}&=-\frac{r_{kj}\{x_{k{i_0}}-a_k u_{22}^{(i_0)}+\lambda_{kj}^+(x_{j{i_0}}-a_j u_{11}^{(i_0)})\}}{\psi(x_{kji_0},x_{kj},a_{i_0})}, \\
	u_{12}^{(i)}&=-\frac{\lambda_{kj}^- x_{i {i_0}}+r_{kj} u_{11}^{(i)} u_{11}^{(i_0)}-x_{kji{i_0}}}{r_{kj}}, \\
	u_{21}^{(i)}&=-\frac{r_{kj}(\lambda_{kj}^+ x_{i {i_0}}-r_{kj} u_{22}^{(i)} u_{22}^{(i_0)}-x_{kji{i_0}})}{\psi(x_{kji_0},x_{kj},a_{i_0})}.
	\end{aligned}
\end{equation}
Here $x_{kji{i_0}}$ is defined by
\begin{equation}\label{Eq_xkjii0}
\begin{split}
x_{kji{i_0}}=&\frac{1}{2}(a_k a_j a_i a_{i_0} +a_k x_{ji{i_0}}+a_j x_{ki{i_0}}+a_i x_{kj{i_0}}+a_{i_0} x_{kji}\\
&+x_{kj} x_{i {i_0}}-x_{ki}x_{j{i_0}}+x_{k{i_0}}x_{ji}  \\
&-a_k a_j x_{i{i_0}}-a_k a_{i_0} x_{ji}-a_j a_i x_{k{i_0}}-a_{i_0} a_i x_{kj}).
\end{split}
\end{equation}
	\end{enumerate}
\end{definition}
%%%%%%%%%%%
%%%%%%%%%

This definition is a generalization of the normal forms by Iwasaki \cite[Definition 3.3]{I2}.
The well-definedness of Definition \ref{Definition_NormalForm} is guaranteed by the following lemma.

%%%%
\begin{lemma}\label{Lemma_welldef}
The map $\varphi_{kj}^{(i_0)}$ is well-defined, that is, the conjugacy class of the tuple 
$(M_1,\ldots,M_n)$ defined in \eqref{Eq_rep_case1} or \eqref{Eq_rep_case2} is uniquely determined, not depending on the choice of the branch in \eqref{r_kj}.
\end{lemma}
%%%%
\begin{proof}
We prove this lemma in a similar manner to Lemma 3.4 in \cite{I2}.
We only consider the case $i_0=0$: the other cases are treated in a similar manner and hence omitted.
Taking the other branch in \eqref{r_kj} has the effect that $r_{kj} \leftrightarrow -r_{kj}$ and $\lambda_{kj}^\pm \leftrightarrow\lambda_{kj}^{\mp}$,
which results in a change of the tuple $(M_1,\ldots,M_n)$:
if we express $M_s=(u_{pq}^{(s)})_{1\le p,q \le 2}$ for $s=1,\ldots,n$, the change is given by
\[
 \begin{pmatrix}
 u_{11}^{(s)} & u_{12}^{(s)} \\
 u_{21}^{(s)} & u_{22}^{(s)}
 \end{pmatrix}
 \quad 
 \leftrightarrow
 \quad 
 \begin{pmatrix}
 u_{22}^{(s)} & -\dfrac{\psi(x_{kj},a_k,a_j)}{r_{kj}^2}u_{21}^{(s)} \\[5pt]
 -\dfrac{r_{kj}^2}{\psi(x_{kj},a_k,a_j)}u_{12}^{(s)} & u_{11}^{(s)}
 \end{pmatrix}.
\]
However, this change is canceled by taking conjugation by a matrix
\[
\begin{pmatrix}
 0 & \xi \\
 -\xi^{-1} & 0
 \end{pmatrix}
 \quad 
 \mathrm{such~that}
 \quad 
 \xi^2=\frac{\psi(x_{kj},a_k,a_j)}{r_{kj}^2}.
\]
Hence the conjugacy class is independent of the choice of the branch.
\end{proof}
%%%%

\begin{definition}
We call the following Zariski open sets
\[
\mathcal{S}^\circ(a)=\bigcup_{(j,k)\in\mathcal{I}_n}\bigcup_{\substack{i_0=0 \\ i_0\neq j,k}}^n \mathcal{S}_{kj}^{(i_0)}(a),\quad 
\mathcal{M}^\circ(a)=\bigcup_{(j,k)\in\mathcal{I}_n}\bigcup_{\substack{i_0=0 \\ i_0\neq j,k}}^n \mathcal{M}_{kj}^{(i_0)}(a)
\]
the {\it big opens}.
\end{definition}
When $n=3$, the big opens $\mathcal{S}^\circ(a)$ and $\mathcal{M}^\circ(a)$ are nothing but the ones defined by Iwasaki \cite{I2}.

\begin{remark}\label{Remark_n=3}
When $n=3$, it holds that $\mathcal{S}^\circ(a)=\mathcal{S}(a)$ and $\mathcal{M}^\circ(a)=\mathcal{M}(a)$ for a generic $a \in \mathbb{C}^4$ (for the precise meaning of generic, see \cite{I2} or \cite{Ada}).
Although it is expected that the same result holds even in the case of general $n$, we have not obtained a proof yet.
\end{remark}

Now we establish a parametrization theorem for $\mathcal{M}^\circ(a)$ in terms of $\mathcal{S}^\circ(a)$, 
which is a generalization of Iwasaki \cite{I2} and Calligaris-Mazzocco \cite{CM}.

%%%%

\begin{theorem}\label{Theorem_Parametrization}
For each $(j,k)\in\mathcal{I}_n$ and $i_0 \in \{0,1,\ldots,n\}\setminus\{j,k\}$, the map
$\varphi_{kj}^{(i_0)}: \mathcal{S}_{kj}^{(i_0)}(a) \to \mathcal{M}_{kj}^{(i_0)}(a)$
in \eqref{Eq_varphi_kj} is bijective. 
These bijections are patched together to yield a global bijection between the big opens,
\begin{equation}\label{Eq_varphi}
\varphi: \mathcal{S}^\circ(a) \to \mathcal{M}^\circ(a).
\end{equation}
\end{theorem}
%%%%

\begin{proof}
Fix $(j,k)\in\mathcal{I}_n$.
We first show that the map
$\varphi_{kj}^{(i_0)}$
 is surjective. 
Given any $[\rho]\in\mathcal{M}_{kj}^{(i_0)}(a)$, we set $x:=\Phi_{kj}^{(i_0)}([\rho])\in\mathcal{S}_{kj}^{(i_0)}$.
Let us show that 
\begin{equation}\label{Eq_surj}
\varphi_{kj}^{(i_0)}(x)=[\rho]
\end{equation}
holds, namely, we can express a representative of $[\rho]$ 
by the form \eqref{Eq_rep_case1} or \eqref{Eq_rep_case2}.
Since $x_{kj} \neq \pm2$ holds now, we have the numbers $\lambda_{kj}^{\pm}$ are distinct. 
Therefore we can take a representative $\rho=(M_1,\ldots,M_n)$ such that
\begin{equation}\label{Eq_MkMj}
M_kM_j=\begin{pmatrix}
\lambda_{kj}^+ & \\
 & \lambda_{kj}^-
 \end{pmatrix}.
\end{equation}
For this representative, we put
\begin{equation}
M_s=\begin{pmatrix} \label{Eq_Ms}
 	u_{11}^{(s)} & u_{12}^{(s)} \\
	u_{21}^{(s)} & u_{22}^{(s)}
\end{pmatrix}
\quad 
(1\le s \le n).
\end{equation}
Then, the equation \eqref{Eq_MkMj} yields
\begin{equation}\label{Eq_ukuj}
 \begin{cases}
 u_{11}^{(k)}= \lambda_{kj}^+ u_{22}^{(j)}, \\[5pt]
 u_{12}^{(k)}=-\lambda_{kj}^+ u_{12}^{(j)}, \\[5pt]
 u_{21}^{(k)}=-\lambda_{kj}^- u_{21}^{(j)}, \\[5pt]
 u_{22}^{(k)}=\lambda_{kj}^- u_{11}^{(j)}.
 \end{cases}
\end{equation}
The conditions $\tr(M_j)=a_j$ and $\tr(M_k)=a_k$ yield
\[
 \begin{cases}
 u_{11}^{(j)}+u_{22}^{(j)}=a_j, \\[5pt]
 \lambda_{kj}^+ u_{22}^{(j)}+\lambda_{kj}^-u_{11}^{(j)}=a_k,
 \end{cases}
\]
where we used \eqref{Eq_ukuj} to derive the second equation.
By solving this system we have
\begin{equation}\label{Eq_u11u22_j}
 u_{11}^{(j)}=-\frac{a_k-\lambda_{kj}^+a_j}{r_{kj}},\quad 
 u_{22}^{(j)}=\frac{a_k-\lambda_{kj}^-a_j}{r_{kj}}.
\end{equation}
Substituting this into \eqref{Eq_ukuj}, we obtain
\begin{equation}\label{Eq_u11u22_k}
 u_{11}^{(k)}=-\frac{a_j-\lambda_{kj}^+a_k}{r_{kj}},\quad 
 u_{22}^{(k)}=\frac{a_j-\lambda_{kj}^-a_k}{r_{kj}}.
\end{equation}
%
%Next we proceed to derive the diagonal elements $u_{11}^{(i)}, u_{22}^{(i)}$ for $i \notin \{j,k\}$.
For $i \in \{1,\ldots,n\}\setminus\{j,k\}$, the conditions $\tr(M_i)=a_i$ and $\tr(M_kM_jM_i)=x_{kji}$ yield
\[
 \begin{cases}
 u_{11}^{(i)}+u_{22}^{(i)}=a_i, \\[5pt]
 \lambda_{kj}^+ u_{11}^{(i)}+\lambda_{kj}^-u_{22}^{(i)}=x_{kji},
 \end{cases}
\]
where we used \eqref{Eq_MkMj} to derive the second equation.
By solving this system, we have
\begin{equation}\label{Eq_u11u22_i}
 u_{11}^{(i)}=\frac{x_{kji}-\lambda_{kj}^-a_i}{r_{kj}},\quad 
 u_{22}^{(i)}=-\frac{x_{kji}-\lambda_{kj}^+a_i}{r_{kj}}.
\end{equation}
The point is that the diagonal entries $\{u_{11}^{(s)},u_{22}^{(s)}\}$ of the representative $\rho=(M_1,\ldots,M_n)$ are determined
independently of $i_0$.

Next we proceed to derive the anti-diagonal entries.
Applying \eqref{Eq_u11u22_j} to the condition $\det(M_j)=1$, we have
\begin{equation}\label{Eq_u12u21_j}
u_{12}^{(j)}u_{21}^{(j)}=-\frac{\psi(x_{kj},a_k,a_j)}{r_{kj}^2}.
\end{equation}
Similarly, applying \eqref{Eq_u11u22_i} to the condition $\det(M_i)=1$, we have
\begin{equation}\label{Eq_u12u21_i}
u_{12}^{(i)}u_{21}^{(i)}=-\frac{\psi(x_{kji},x_{kj},a_i)}{r_{kj}^2}.
\end{equation}
From now on, we shall divide into two cases. 
%%%%%%%case1
\medskip \\
\noindent
\textbf{Case 1.} If $[\rho] \in \mathcal{M}_{kj}^{(0)}(a)$, we see that $u_{12}^{(j)}u_{21}^{(j)}\neq0$ holds from \eqref{Eq_u12u21_j}.
In this case, we can normalize $M_j$ so that
\begin{equation}\label{Eq_u21_u12_j}
 u_{21}^{(j)}=1, \quad
 u_{12}^{(j)}=-\frac{\psi(x_{kj},a_k,a_j)}{r_{kj}^2}
\end{equation}
by the conjugation $P\rho P^{-1}$ using a suitable diagonal matrix $P$. 
This conjugation does not change the diagonal entries of all $M_s$.
Hence we can assume from the beginning that \eqref{Eq_u21_u12_j} is satisfied.
Then we obtain $M_j$ and $M_k$ in \eqref{Eq_rep_case1}.

Let us derive $u_{12}^{(i)},u_{21}^{(i)}$ for each $i \in \{1,2,\ldots,n\}\setminus\{j,k\}$.
The conditions $\tr(M_j M_i)=x_{ji}$ and $\tr(M_k M_i)=x_{ki}$ yield
\[
 \begin{cases}
 u_{11}^{(j)}u_{11}^{(i)}+u_{12}^{(j)}u_{21}^{(i)}+u_{21}^{(j)}u_{12}^{(i)}+u_{22}^{(j)}u_{22}^{(i)}=x_{ji}, \\[5pt]
 \lambda_{kj}^+u_{22}^{(j)}u_{11}^{(i)}-\lambda_{kj}^+ u_{12}^{(j)}u_{21}^{(i)}-\lambda_{kj}^- u_{21}^{(j)}u_{12}^{(i)}+\lambda_{kj}^- u_{11}^{(j)}u_{22}^{(i)}=x_{ki},
 \end{cases}
\]
where we used \eqref{Eq_ukuj} to obtain the second equation.
By substituting \eqref{Eq_u11u22_j} and \eqref{Eq_u21_u12_j} into this system and solving it, we have \eqref{Eq_u12_u21_i_case1}.
Comparing above results and \eqref{Eq_rep_case1} with \eqref{Eq_M_i_diag} and \eqref{Eq_u12_u21_i_case1}, we conclude that $\varphi_{kj}^{(0)}(x)=[\rho]$, 
namely, the map $\varphi_{kj}^{(0)}$ is surjective.
%
%%%%%%%case2
\medskip \\
\noindent
\textbf{Case 2.}
If $[\rho]\in\mathcal{M}_{kj}^{(i_0)}(a)$ for some $i_0 \in \{1,2,\ldots,n\}\setminus\{j,k\}$, thanks to \eqref{Eq_u12u21_i}, we can normalize $M_{i_0}$ so that
\begin{equation} \label{Eq_u21_u12_i0_case2}
u_{21}^{(i_0)}=1,\quad u_{12}^{(i_0)}=-\frac{\psi(x_{kj{i_0}},x_{kj},a_{i_0})}{r_{kj}^2}
\end{equation}
for the same reason as Case 1.

We proceed to determine $u_{12}^{(j)}$ and $u_{21}^{(j)}$.
The conditions $\tr(M_j M_{i_0})=x_{ji_0}$ and $\tr(M_k M_{i_0})=x_{k{i_0}}$ yield
\[
 \begin{cases}
 u_{11}^{(j)}u_{11}^{(i_0)}+u_{12}^{(j)}u_{21}^{(i_0)}+u_{21}^{(j)}u_{12}^{(i_0)}+u_{22}^{(j)}u_{22}^{(i_0)}=x_{ji_0}, \\[5pt]
 \lambda_{kj}^+u_{22}^{(j)}u_{11}^{(i_0)}-\lambda_{kj}^+ u_{12}^{(j)}u_{21}^{(i_0)}-\lambda_{kj}^- u_{21}^{(j)}u_{12}^{(i_0)}+\lambda_{kj}^- u_{11}^{(j)}u_{22}^{(i_0)}=x_{ki_0},
 \end{cases}
\]
where we used \eqref{Eq_ukuj} to obtain the second equation.
Substituting \eqref{Eq_u11u22_i} with $i=i_0$ and \eqref{Eq_u21_u12_i0_case2} into this system and solving it, we have $u_{12}^{(j)}$ and $u_{21}^{(j)}$ in \eqref{Eq_u12u21_ij_case2}.

Next, we proceed to derive $u_{12}^{(i)},u_{21}^{(i)}$ for $i \notin \{i_0,j,k\}$.
From \cite{V, ABL}, we have the following. 
%%%%
\begin{lemma}\label{Lemma_reduction}
Let $A,B,C,D \in \mathrm{SL}(2,\mathbb{C})$. 
Then it holds that
\[
\begin{split}
\tr(ABCD)&=\frac{1}{2}\{\tr(A) \tr(B) \tr(C) \tr(D) \\
&+\tr(A)\tr(BCD) + \tr(B)\tr(ACD)+\tr(C) \tr(ABD)+\tr(D)\tr(ABC) \\
&+\tr(AB) \tr(CD) -\tr(AC)\tr(BD) +\tr(AD)\tr(BC) \\
&-\tr(A)\tr(B)\tr(CD)-\tr(A)\tr(D)\tr(BC)\\
&-\tr(B) \tr(C) \tr(AD)-\tr(D) \tr(C) \tr(AB)\}.
\end{split}
\]
\end{lemma}
Thanks to this lemma, we have $\tr(M_kM_jM_iM_{i_0})=x_{kji{i_0}}$, where $x_{kji{i_0}}$ is given in \eqref{Eq_xkjii0}.
Now, the conditions $\tr(M_iM_{i_0})=x_{i{i_0}}$ and $\tr(M_kM_jM_iM_{i_0})=x_{kji{i_0}}$ yield
\[
 \begin{cases}
 u_{11}^{(i)}u_{11}^{(i_0)}+u_{12}^{(i)}u_{21}^{(i_0)}+u_{21}^{(i)}u_{12}^{(i_0)}+u_{22}^{(i)}u_{22}^{(i_0)}=x_{ii_0}, \\[5pt]
 \lambda_{kj}^+(u_{11}^{(i)}u_{11}^{(i_0)}+u_{12}^{(i)}u_{21}^{(i_0)})
 +\lambda_{kj}^-(u_{21}^{(i)}u_{12}^{(i_0)}+u_{22}^{(i)}u_{22}^{(i_0)})=x_{kji{i_0}},
 \end{cases}
\]
where we used \eqref{Eq_MkMj} to obtain the second equation.
Substituting \eqref{Eq_u21_u12_i0_case2} into this system and solving it, we have $u_{12}^{(i)}$ and $u_{21}^{(i)}$ in \eqref{Eq_u12u21_ij_case2}.
Comparing above results and \eqref{Eq_rep_case2} with \eqref{Eq_u11u22_i_case2} and \eqref{Eq_u12u21_ij_case2}, we conclude that $\varphi_{kj}^{(i_0)}(x)=[\rho]$, 
namely, the map $\varphi_{kj}^{(i_0)}$ is surjective.

To show that the map $\varphi_{kj}^{(i_0)}$ is injective, we only notice that, once the normalizations \eqref{Eq_u21_u12_j} or \eqref{Eq_u21_u12_i0_case2} are employed, the admissible simultaneous similarity transformations for $(M_1,\ldots,M_n)$ are only the ones by $\pm I_n$.
However, such transformation leaves every entry in $(M_1,\ldots,M_n)$ unchanged.
This fact implies the injectivity of the map $\varphi_{kj}^{(i_0)}$. 
\end{proof}
The equality \eqref{Eq_surj} in the above proof implies that 
\[
\varphi_{kj}^{(i_0)}\circ\Phi_{kj}^{(i_0)}=\mathrm{id}_{\mathcal{M}_{kj}^{(i_0)}}.
\]
From the bijectivity, we see that the map $\varphi_{kj}^{(i_0)}(x)$ is the inverse of the restricted map $\Phi_{kj}^{(i_0)}$ in \eqref{Eq_Phi_rest}.
Namely, the map $\varphi$ in \eqref{Eq_varphi} is a ``global" inverse mapping of $\Phi \vert_{\mathcal{M}^{\circ}(a)}$.
%%%%

%%%%%%%%%%%%%%%%
%%%%%%%%%%%%%%%%
\section{Unitary monodromies in big open}
\label{Section_characterization}
%%%%%%%%%%%%%%%%
%%%%%%%%%%%%%%%%

In this section we study the unitarity and signatures of monodromies in $\mathcal{M}^\circ(a)$ 
in terms of $\mathcal{S}^\circ(a)$.
To show the monodromy $[\rho]\in\mathcal{M}(a)^{irr}$ is unitary, 
thanks to Proposition \ref{Proposition_rep_Herm}, 
we have only to show that one of the
representative has an invariant non-degenerate Hermitian form.
Therefore in this section, we consider the representative constructed in Definition \ref{Definition_NormalForm} and give a criterion for it to have an invariant Hermitian form.
To achieve this, we first prepare an important lemma.
%%%%
%%%%
\begin{lemma}[\cite{Ada}, Corollary 4.3]\label{Lemma_Necessory} 
Let $M$ be a matrix in $\mathrm{SL}(2,\mathbb{C})$.
If there is a non-degenerate 
Hermitian matrix $H$ satisfying 
$\bar{M}^T H M=H$,
then $\tr M \in \mathbb{R}$ holds.
\end{lemma}
%%%%
%%%%
As a consequence of this lemma, we see that the assumption
\begin{equation}\label{Eq_a_R}
a=(a_1,a_2\ldots,a_n,a_{n+1}) \in \mathbb{R}^{n+1}
\end{equation}
is needed for the existence of unitary monodromies in $\mathcal{M}(a)^{irr}$.
Also, we have to assume $x=\Phi([\rho])\in \mathcal{S}^\circ(a)\cap\mathbb{R}^m$ for the unitarity of $[\rho] \in \mathcal{M}(a)^{irr}$.
The following proposition states that this condition is also a sufficient condition for the existence of an invariant Hermitian form.
%%

%%%%
\begin{proposition}\label{Prop_UnitaryMonodromy}
Assume \eqref{Eq_a_R} and fix the bijection 
$\varphi:\mathcal{S}^\circ(a) \to \mathcal{M}^\circ(a)$
by \eqref{Eq_varphi}.
Then the element
$\varphi(x)=[(M_1,\ldots,M_n)] \in \mathcal{M}^\circ(a)$
has an invariant non-degenerate Hermitian form if and only if
$x \in \mathcal{S}^\circ(a)\cap\mathbb{R}^m$.  
 If this is the case, the followings hold.
\begin{enumerate}
\item Let $x\in \mathcal{S}_{kj}^{(0)}(a)\cap{\mathbb{R}^m}$ and $(M_1,\ldots,M_n)$ be the representative of $\varphi(x)$ given by \eqref{Eq_rep_case1}.
Then the invariant non-degenerate Hermitian matrix $H$ satisfying \eqref{eq_invHerm} is given by
	 \begin{equation}\label{Eq_invHerm_case1}
	 H=
	   \begin{cases} 
	    h \begin{pmatrix}
	     1 & \\
	     & \dfrac{\psi(x_{kj},a_k,a_j)}{x_{kj}^2-4}
	     \end{pmatrix}  & (x_{kj}^2-4<0), \\[22pt]
	   h  \begin{pmatrix}
	      & \sqrt{-1} \\
	      -\sqrt{-1} &
	     \end{pmatrix} & (x_{kj}^2-4>0).
	   \end{cases}
	\end{equation}
\item Let $x \in \mathcal{S}_{kj}^{(i_0)}(a)\cap\mathbb{R}^m$ for some $i_0 \in \{1,2,\ldots, n\}\setminus\{j,k\}$ and $(M_1,\ldots,M_n)$ be the representative of $\varphi(x)$ given by \eqref{Eq_rep_case2}.
Then the invariant non-degenerate Hermitian matrix $H$ satisfying \eqref{eq_invHerm} is given by
	 \begin{equation} \label{Eq_invHerm_case2}
	 H=
	   \begin{cases} 
	   h  \begin{pmatrix}
	     1 & \\
	     & \dfrac{\psi(x_{kji_0},x_{kj},a_{i_0})}{x_{kj}^2-4}
	     \end{pmatrix}  & (x_{kj}^2-4<0), \\[22pt]
	    h \begin{pmatrix}
	      & \sqrt{-1} \\
	      -\sqrt{-1} &
	     \end{pmatrix}  &(x_{kj}^2-4>0),
	   \end{cases}
	  \end{equation}
	\end{enumerate}
where $h$ is an arbitrary non-zero real number.
These Hermitian matrices are uniquely determined up to a multiplication of real numbers.
\end{proposition}
%%%%
\begin{proof}
At first we note that,
thanks to Lemma \ref{Lemma_tr}, 
the condition $x\in\mathbb{R}^m$ is equivalent to the condition 
$x_{kj}, x_{kji} \in \mathbb{R}$ for all $\{i,j,k\}\subset \{1,2,\ldots,n\}$.
Therefore ``only if" part follows from Lemma \ref{Lemma_Necessory} immediately.
So we shall show the ``if" part.
We only show the case (i): the other case (ii) can be treated in the same manner.
We assume $x\in\mathcal{S}_{kj}^{(0)}(a)\cap\mathbb{R}^m$
and let $(M_1,M_2,\ldots,M_n)$ be a tuple of matrices given in \eqref{Eq_rep_case1}.
We set
\[
\Lambda_{kj}:=M_kM_j=
 \begin{pmatrix}
   \lambda_{kj}^+ & \\
    & \lambda_{kj}^-
  \end{pmatrix}.
\]
Our goal is to determine the matrix $H$ satisfying
\begin{equation}\label{Eq_invHerm_proof}
\bar{M}_s^T HM_s=H \quad (s=1,2,\ldots,n).
\end{equation}
Combining the equations $\bar{M}_j^T HM_j=H$ and $\bar{M}_k^T HM_k=H$,
we have
\begin{equation}\label{Eq_LambdaH}
\bar{\Lambda}_{kj}H\Lambda_{kj}=H,
\end{equation}
where we used $\Lambda_{kj}=\Lambda_{kj}^T$.

We set
\[ 
  H=\begin{pmatrix}
  h_{11} & h_{12} \\
  h_{21} & h_{22}
  \end{pmatrix}
\]
and then obtain
\begin{equation}\label{Eq_LHL}
  \bar{\Lambda}_{kj}H\Lambda_{kj}=
   \begin{pmatrix}
    |\lambda_{kj}^+|^2 h_{11} & \overline{\lambda_{kj}^+} \lambda_{kj}^- h_{12} \\[6pt]
    \overline{\lambda_{kj}^-} \lambda_{kj}^+ h_{21} & |\lambda_{kj}^-|^2 h_{22}
   \end{pmatrix}.
\end{equation}
%
%%%%%%%case1
\medskip \\
\noindent
\textbf{Case 1.} 
We consider the case $x_{kj}^2-4>0$.
This yields 
$r_{kj}\in\mathbb{R}\setminus\{0\}$ and $\lambda_{kj}^\pm\in\mathbb{R}\setminus\{0,\pm1\}$.
Therefore we have $M_i \in \mathrm{SL}(2,\mathbb{R})\,(i=1,2,\ldots,n)$ and 
\[
\bar{\Lambda}_{kj}H\Lambda_{kj}=
   \begin{pmatrix}
    |\lambda_{kj}^+|^2 h_{11} & \lambda_{kj}^+ \lambda_{kj}^- h_{12} \\[6pt]
    \lambda_{kj}^- \lambda_{kj}^+ h_{21} & |\lambda_{kj}^-|^2 h_{22}
   \end{pmatrix}
 =
   \begin{pmatrix}
    |\lambda_{kj}^+|^2 h_{11} &  h_{12} \\[6pt]
     h_{21} & |\lambda_{kj}^-|^2 h_{22}
   \end{pmatrix}.
\]
Since $|\lambda_{kj}^\pm|\neq 1,0$ holds, we have
\[
h_{11}=h_{22}=0
\]
from the equation \eqref{Eq_LambdaH}.
Next we proceed to determine $h_{12}$ and $h_{21}$.
Let us substitute
\[
H=\begin{pmatrix}
  0 & h_{12} \\
  h_{21} & 0
\end{pmatrix}
\]
into the equation $\bar{M}_{j}^T HM_{j}=H$.
Now we forget the actual expressions of $M_j$ and 
write as \eqref{Eq_Ms}, that is, we express
\[
M_j=\begin{pmatrix}
 u_{11}^{(j)} & u_{12}^{(j)}\\
 u_{21}^{(j)} & u_{22}^{(j)}
 \end{pmatrix}.
\]
Then we have
\begin{equation}\label{Eq_LastMat0_case1}
\begin{aligned}
\bar{M}_{j}^T HM_{j} &= M_j^THM_j
= \begin{pmatrix}
  u_{21}^{(j)}u_{11}^{(j)}(h_{21}+h_{12}) & u_{21}^{(j)}u_{12}^{(j)}h_{21}+u_{11}^{(j)}u_{22}^{(j)}h_{12} \\
  u_{11}^{(j)}u_{22}^{(j)}h_{21}+u_{12}^{(j)}u_{21}^{(j)}h_{12} & u_{12}^{(j)}u_{22}^{(j)}(h_{21}+h_{12})
  \end{pmatrix},
\end{aligned}
\end{equation}
where we used $M_j \in\mathrm{SL}(2,\mathbb{R})$ in the first equality.
Therefore the equation $\bar{M}_{j}^T HM_{j}=H$ becomes
\begin{equation}\label{Eq_LastMat_case1}
\begin{pmatrix}
  u_{21}^{(j)}u_{11}^{(j)}(h_{21}+h_{12}) & u_{21}^{(j)}u_{12}^{(j)}h_{21}+u_{11}^{(j)}u_{22}^{(j)}h_{12} \\
  u_{11}^{(j)}u_{22}^{(j)}h_{21}+u_{12}^{(j)}u_{21}^{(j)}h_{12} & u_{12}^{(j)}u_{22}^{(j)}(h_{21}+h_{12})
  \end{pmatrix}
  =
\begin{pmatrix}
  0 & h_{12} \\
  h_{21} & 0
\end{pmatrix}.
\end{equation}
Now we focus on the equation for the $(1,2)$-entry
\begin{equation}\label{Eq_(1,2)}
u_{21}^{(j)}u_{12}^{(j)}h_{21}+u_{11}^{(j)}u_{22}^{(j)}h_{12}=h_{12}.
\end{equation}
The condition $\det(M_j)=1$ yields $u_{11}^{(j)}u_{22}^{(j)}-u_{12}^{(j)}u_{21}^{(j)}=1$.
Substituting this into \eqref{Eq_(1,2)}, we have
\[
(h_{12}+h_{21})u_{12}^{(j)}u_{21}^{(j)}=0.
\] 
Since $[\rho] \in \mathcal{M}_{kj}^{(0)}(a)$ now, it holds that $u_{12}^{(j)}u_{21}^{(j)}\neq0$
(cf. the proof of Theorem \ref{Theorem_Parametrization}).
Therefore we have $h_{21}=-h_{12}$ and this satisfies the other equations in \eqref{Eq_LastMat_case1}. 
Then we have 
\begin{equation}\label{Eq_H_case1_0}
H=\begin{pmatrix}
   0 & h_{12} \\
   -h_{12} & 0 
\end{pmatrix}
\end{equation}
is a solution for $\bar{M}_{j}^T HM_{j}=H$.
Here we note that the actual expressions of the entries of $M_j$ were not used in the above calculation \eqref{Eq_LastMat0_case1} and \eqref{Eq_LastMat_case1}. 
Therefore, if we write 
$M_s ~(s\in\{1,2,\ldots,n\}\setminus\{j\})$
in the form \eqref{Eq_Ms},
then the equation $\bar{M}_{s}^T HM_{s}=H$ can be expressed in the same form with \eqref{Eq_LastMat_case1}.
From $h_{21}=-h_{12}$, the equality $\bar{M}_{s}^T HM_{s}=H$ immediately follows.
This means that \eqref{Eq_H_case1_0} is a solution of the equations \eqref{Eq_invHerm_proof} in the case $x_{kj}^2-4>0$.
Since $h_{12}$ can be taken arbitrarily, we set $h_{12}=h\sqrt{-1}~(h \in \mathbb{R})$ and obtain the assertion for $x_{kj}^2-4>0$.
%
%%%%%%%case2
\medskip \\
\noindent
\textbf{Case 2.}
We consider the case $x_{kj}^2-4<0$. 
In this case, the following lemma holds.
%%%%
\begin{lemma}
If $x_{kj}^2-4<0$, then $\lambda_{kj}^\pm\notin\mathbb{R}$ and 
$|\lambda_{kj}^+|=|\lambda_{kj}^-|=1$.
\end{lemma}
%%%%
\begin{proof}
Since $x_{kj}^2-4<0$ holds now, by the definition \eqref{r_kj},
the number $r_{kj}$ is a non-zero purely imaginary number.
Then the first assertion is clear. 
In this case it holds that $\overline{\lambda_{kj}^+}=\lambda_{kj}^-$, then we have
\[
|\lambda_{kj}^+|^2=\lambda_{kj}^+\overline{\lambda_{kj}^+}=\lambda_{kj}^+\lambda_{kj}^-=1.
\]
We can show $|\lambda_{kj}^-|^2=1$ in a similar manner.
\end{proof}
%%%%
With this lemma in mind, we have
\[
  \bar{\Lambda}_{kj}H\Lambda_{kj}=
   \begin{pmatrix}
    h_{11} & (\lambda_{kj}^-)^2 h_{12} \\[6pt]
    (\lambda_{kj}^+)^2 h_{21} &  h_{22}
   \end{pmatrix}
\]
from \eqref{Eq_LHL}.
Therefore, the equation \eqref{Eq_LambdaH} turns into
\[
\begin{pmatrix}
    h_{11} & (\lambda_{kj}^-)^2 h_{12} \\[6pt]
    (\lambda_{kj}^+)^2 h_{21} &  h_{22}
   \end{pmatrix}
   =
\begin{pmatrix}
  h_{11} & h_{12} \\
  h_{21} & h_{22}
\end{pmatrix}.
\]
Since $\lambda_{kj}^\pm \notin \{0,\pm1\}$ holds now, we have
\[
h_{12}=h_{21}=0.
\]
Next, we proceed to determine $h_{11}$ and $h_{22}$. 
To consider the equations \eqref{Eq_invHerm_proof}, 
we give a simple but important observation for the entries of $M_s\,(s=1,2,\ldots,n)$.
%%%%
\begin{lemma}\label{Lemma_symmetry}
Assume $(a,x)\in\mathbb{R}^{n+1}\times (\mathcal{S}^\circ(a)\cap\mathbb{R}^m)$ and 
$x_{kj}^2-4<0$. 
For $[\rho] \in \mathcal{M}_{kj}^{(0)}(a)$, let \eqref{Eq_Ms} be the representative \eqref{Eq_rep_case1}.
Then, 
\begin{align}
\bar{u}_{11}^{(s)}&=u_{22}^{(s)}, \label{Eq_Lemma_symmetry1} \\
\frac{u_{12}^{(s)}}{\bar{u}_{21}^{(s)}}&=-\frac{\psi(x_{kj},a_k,a_j)}{x_{kj}^2-4}  \label{Eq_Lemma_symmetry2}
\end{align}
hold for any $s\in\{1,2,\ldots,n\}$.
\end{lemma}
%%%%
This lemma can be shown by applying $\bar{r}_{kj}=-r_{kj}$ and
$\overline{\lambda_{kj}^+}=\lambda_{kj}^-$
to \eqref{Eq_rep_case1}, 
\eqref{Eq_M_i_diag} and \eqref{Eq_u12_u21_i_case1}.
Let us consider the equations \eqref{Eq_invHerm_proof}.
As in the above case, we write
\[
M_j=\begin{pmatrix}
 u_{11}^{(j)} & u_{12}^{(j)}\\
 u_{21}^{(j)} & u_{22}^{(j)}
 \end{pmatrix}
\]
and substitute this into the equation $\bar{M}_j^T H M_j=H$.
Then we have
\begin{equation}\label{Eq_LastMat_case2}
\begin{pmatrix}
  u_{11}^{(j)}u_{22}^{(j)} h_{11}+|u_{21}^{(j)}|^2 h_{22} & u_{12}^{(j)} u_{22}^{(j)} h_{11}+\bar{u}_{21}^{(j)}u_{22}^{(j)} h_{22} \\
  \bar{u}_{12}^{(j)}u_{11}^{(j)}h_{11}+u_{11}^{(j)}u_{21}^{(j)}h_{22} & |u_{12}^{(j)}|^2h_{11}+u_{11}^{(j)}u_{22}^{(j)}h_{22}
  \end{pmatrix}
  =
\begin{pmatrix}
  h_{11} & 0 \\
  0 & h_{22}
\end{pmatrix},
\end{equation}
where we used \eqref{Eq_Lemma_symmetry1}.
We focus on the equation for the $(1,1)$-entry
\[
u_{11}^{(j)}u_{22}^{(j)} h_{11}+|u_{21}^{(j)}|^2 h_{22}=h_{11}.
\]
From $u_{11}^{(j)}u_{22}^{(j)}-u_{12}^{(j)}u_{21}^{(j)}=1$,
this equation becomes
\begin{equation}\label{Eq_(1,1)entry}
u_{21}^{(j)}(u_{12}^{(j)} h_{11}+\bar{u}_{21}^{(j)} h_{22})=0.
\end{equation}
Since $[\rho] \in \mathcal{M}_{kj}^{(0)}(a)$ now, it holds that $u_{12}^{(j)}u_{21}^{(j)}\neq0$.
Therefore, we can solve \eqref{Eq_(1,1)entry} with respect to $h_{22}$.
Then we obtain
\begin{equation}\label{Eq_h22_case2}
h_{22}=-\frac{u_{12}^{(j)}}{\bar{u}_{21}^{(j)}}h_{11}=\frac{\psi(x_{kj},a_k,a_j)}{x_{kj}^2-4} h_{11},
\end{equation}
where we used \eqref{Eq_Lemma_symmetry2} in the second equality.
It is easy to see that \eqref{Eq_h22_case2} satisfies the other equations in \eqref{Eq_LastMat_case2}.
Therefore,
\begin{equation}\label{Eq_H_case2_0}
H=\begin{pmatrix}
   h_{11} & 0 \\[5pt]
   0 & -\dfrac{u_{12}^{(j)}}{\bar{u}_{21}^{(j)}}h_{11} 
\end{pmatrix}
  =\begin{pmatrix}
   h_{11} & 0 \\[5pt]
   0 & \dfrac{\psi(x_{kj},a_k,a_j)}{x_{kj}^2-4} h_{11} 
\end{pmatrix}
\end{equation}
is a solution for the equation $\bar{M}_j^T H M_j=H$.
Thanks to Lemma \ref{Lemma_symmetry}, we can check that \eqref{Eq_H_case2_0} satisfies the other equations in \eqref{Eq_invHerm_proof} as in Case 1.
This means that \eqref{Eq_H_case2_0} is a solution of the equations \eqref{Eq_invHerm_proof} in the case $x_{kj}^2-4<0$.
Since $h_{11}$ can be taken arbitrarily, we set $h_{11}=h~(h \in \mathbb{R})$ and obtain the assertion for $x_{kj}^2-4<0$.
This completes the proof.
\end{proof}
For reader's convenience, we give a lemma used to consider the case (ii) of Proposition \ref{Prop_UnitaryMonodromy},
which corresponds to Lemma \ref{Lemma_symmetry}.
%%%%
\begin{lemma}
Assume $(a,x)\in\mathbb{R}^{n+1}\times (\mathcal{S}^\circ(a)\cap\mathbb{R}^m)$ and 
$x_{kj}^2-4<0$. 
For $[\rho] \in \mathcal{M}_{kj}^{(i_0)}(a)$ with some $i_0 \in \{1,2,\ldots, n\}\setminus\{j,k\}$, let \eqref{Eq_Ms} be the representative \eqref{Eq_rep_case2}.
Then 
\begin{align*}
\bar{u}_{11}^{(s)}&=u_{22}^{(s)},  \\
\frac{u_{12}^{(s)}}{\bar{u}_{21}^{(s)}}&=-\frac{\psi(x_{kji_0},x_{kj},a_{i_0})}{x_{kj}^2-4} 
\end{align*}
hold for all $s\in\{1,2,\ldots,n\}$.
\end{lemma}
%%%%

%Main theorem
By considering the signatures of the invariant Hermitian forms given in Proposition \ref{Prop_UnitaryMonodromy}, we determine the signatures of the unitary monodoromies in $\mathcal{M}^\circ(a)$.
Here we note that, as seen in Remark \ref{Rem_invHerm}, the signature of the monodromy in $\mathcal{M}^\circ(a)$
does not depend on the choice of chart $\mathcal
{M}_{kj}^{(i)}(a)$, to which it belongs.
Summarizing these observations, we obtain the following theorem.
%%%%
\begin{theorem}\label{Theorem_main}
 Assume \eqref{Eq_a_R} and fix the bijection 
 $\varphi:\mathcal{S}^\circ(a) \to \mathcal{M}^\circ(a)$
 by \eqref{Eq_varphi}.
 Then the monodromy $\varphi(x)$ is unitary if and only if $x \in \mathcal{S}^\circ(a)\cap\mathbb{R}^m$.
 If this is the case, the signature of $\varphi(x)$ is determined as follows.
 \begin{itemize}
 	\item If $ x \in \mathcal{S}_{kj}^{(0)}(a)\cap\mathbb{R}^m$ for some $(j,k)\in\mathcal{I}_n$, then
		\begin{itemize}
			  \item the signature is $(1,1)$ if and only if $x_{kj}^2-4>0$, or 
			  $x_{kj}^2-4<0$ and $\psi(x_{kj},a_k,a_j)>0$.
		 	 \item the signature is $(2,0)$ or $(0,2)$ if and only if $x_{kj}^2-4<0$ and $\psi(x_{kj},a_k,a_j)<0$. 
		\end{itemize}
	\item  If $x \in \mathcal{S}_{kj}^{(i_0)}(a)\cap\mathbb{R}^m$ for some $(j,k)\in\mathcal{I}_n$ and $i_0\in \{1,2,\ldots,n\}\setminus\{j,k\}$, then
		\begin{itemize}
			  \item the signature is $(1,1)$ if and only if $x_{kj}^2-4>0$, or 
			  $x_{kj}^2-4<0$ and $\psi(x_{kj{i_0}},x_{kj},a_{i_0})>0$.
		 	 \item the signature is $(2,0)$ or $(0,2)$ if and only if $x_{kj}^2-4<0$ and $\psi(x_{kj{i_0}},x_{kj},a_{i_0})<0$.
		  \end{itemize}
 \end{itemize}
 Here the subset $\mathcal{I}_n$ is given by \eqref{Eq_I_n}.
 The signature of the monodromy $\varphi(x)$ does not depend on the choice of chart $\mathcal
{S}_{kj}^{(i)}(a)$, to which $x\in\mathcal{S}^\circ(a)$ belongs.
\end{theorem}

\appendix
\section{Derivation of Type 3 relation in Proposition \ref{Prop_invariants}}
\label{Appendix}

In this appendix, we give a detailed explanation for Remark \ref{remark_type3}, that is, we show that the condition $f_r(x)=0$ is equal to the condition $M_{n+1} \in \mathcal{O}(a_{n+1})$, by using Lemma \ref{Lemma_reduction}. 
Since $\tr(M_{n+1})=\tr(M_nM_{n-1}\cdots M_1)$ holds, it is sufficient to show that the condition $f_r(x)=0$ is equal to the equation $\tr(M_nM_{n-1}\cdots M_1)=a_{n+1}$.
Let us fix $n \ge 4$ and take indices $i_1,i_2,\ldots,i_k$ such that
\[
1\le i_1 <i_2<\cdots<i_k \le n.
\]
For these indices, we set
\[
M_{i_k,i_l}:=M_{i_k} M_{i_{k-1}} \cdots M_{i_l} \quad (1\le l \le k).
\]
Then for $k=1,2,3$, we have $\tr(M_{i_1,i_1})=\tr(M_{i_1})$, $\tr(M_{i_2,i_1})=\tr(M_{i_2}M_{i_1})$ and $\tr(M_{i_3,i_1})=\tr(M_{i_3}M_{i_2} M_{i_1})$, respectively.
Next, we proceed to consider $\tr(M_{i_k,i_1})$ for $k \ge 4$.
By setting 
\[
A=M_{i_k,i_{4}},\quad B=M_{i_{3}},\quad C=M_{i_2}, \quad D=M_{i_1}
\]
in Lemma \ref{Lemma_reduction}, we have
\begin{equation}\label{App_Reduction}
\begin{split}
\tr (M_{i_k,i_1}) =&\frac{1}{2}\{\tr(M_{i_k,{i_4}}) \tr(M_{i_3}) \tr(M_{i_2}) \tr(M_{i_1}) \\
&+\tr(M_{i_k,{i_4}})\tr(M_{i_3}M_{i_2}M_{i_1}) + \tr(M_{i_3})\tr(M_{i_k,{i_4}}M_{i_2}M_{i_1}) \\
&+\tr(M_{i_2}) \tr(M_{i_k,{i_4}}M_{i_3}M_{i_1})+\tr(M_{i_1})\tr(M_{i_k,{i_4}}M_{i_3}M_{i_2}) \\
&+\tr(M_{i_k,{i_3}}) \tr(M_{i_2}M_{i_1}) -\tr(M_{i_k,{i_4}}M_{i_2})\tr(M_{i_3}M_{i_1}) +\tr(M_{i_k,{i_4}}M_{i_1})\tr(M_{i_3}M_{i_2}) \\
&-\tr(M_{i_k,{i_4}})\tr(M_{i_3})\tr(M_{i_2}M_{i_1})-\tr(M_{i_k,{i_4}})\tr(M_{i_1})\tr(M_{i_3}M_{i_2})\\
&-\tr(M_{i_3}) \tr(M_{i_2}) \tr(M_{i_k,{i_4}}M_{i_1})-\tr(M_{i_1}) \tr(M_{i_2}) \tr(M_{i_k,{i_3}})\}.
\end{split}
\end{equation}
Thanks to this relation, we can express the trace of the product of $k$ matrices by using the trace of the product of at most $k-1$ matrices.
Hence, we can determine $\tr(M_{i_k,i_1})$ for $k\ge 4$ inductively.
As a consequence, we can express $\tr(M_{i_n,i_1})=\tr(M_nM_{n-1}\cdots M_1)$ in terms of 
$\tr(M_{i_1,i_1})$, $\tr(M_{i_2,i_1})$ and $\tr(M_{i_3,i_1})$.

Noting that the structure of the recurrence relation \eqref{App_Reduction} is the same with the one \eqref{Poly_g} for the polynomials $g_{i_k,i_{k-1},\ldots,i_1}(x)$,
we see that $g_{i_n,i_{n-1},\ldots,i_1}(x)=g_{n,n-1,\ldots,1}(x)$ expresses $\tr(M_nM_{n-1}\cdots M_1)$ 
\if0 in terms of
  \[
  \begin{aligned}
  g_{i_1}(x)=a_{i_1} \quad &(1\le i_1\le n), \\
  g_{i_2,i_1}(x)=x_{i_2i_1} \quad &(1\le i_1<i_2\le n),\\
  g_{i_3,i_2,i_1}(x)=x_{i_3i_2i_1} \quad &(1\le i_1<i_2<i_3\le n)
  \end{aligned}
  \]
  \fi
under the assignments $a_{i_1}=\tr(M_{i_1,i_1})$, 
$x_{i_2i_1}=\tr(M_{i_2,i_1})$ and 
$x_{i_3i_2i_1}=\tr(M_{i_3,i_1})$.

As a result, we see that the condition $\tr(M_{n+1})=\tr(M_nM_{n-1}\cdots M_1)=a_{n+1}$ is equal to $f_r(x)=0$.

%%%%%%%%%%%%%%%%%%%%%%%%%%%%
%%%%%%%%%%%%%%%%%%%%%%%%%%%%
\section*{Acknowledgement}

The author would like to express his gratitude to Professor Yoshishige Haraoka for many valuable comments and suggestions for this work and manuscript.
The author also wishes to express his thanks to Professor Saiei-Jaeyeong Matsubara-Heo for reading manuscript and helpful comments.
The author is also grateful to Prof. Hironobu Kimura and Prof. Hiroshi Ogawara for helpful discussions.
Lastly, the author thank the anonymous referee for many valuable comments, which make this paper more readable.
This work was supported by JSPS KAKENHI Grant Number 22J12306.
%%%%%%%%%%%%%%%%%%%%%%%%%%%%
%%%%%%%%%%%%%%%%%%%%%%%%%%%%

%%%References%%%%

\medskip\noindent
Shunya Adachi\\*
Graduate School of Science and Technology\\*
Kumamoto University\\*
Kumamoto 860-8555\\*
Japan\\*
{\tt 200d7101@st.kumamoto-u.ac.jp}\\*

\noindent
Current address: \\*
Department of Mathematics and Informatics\\*
Faculty of Science\\*
Chiba University\\*
Chiba 263-8522 \\*
Japan\\*
{\tt sadachi@math.s.chiba-u.ac.jp}\\*

\end{document}